\documentclass[12pt]{amsart}
\usepackage[margin=1in,left=0.85in,right=0.85in]{geometry}
\usepackage{amsmath,amssymb,amsfonts,amsthm,mathtools}
\usepackage{hyperref}

\DeclareMathOperator{\Av}{Av}
\DeclareMathOperator{\ex}{ex}
\DeclareMathOperator{\cont}{cont}

\newtheorem{theorem}{Theorem}
\newtheorem{proposition}[theorem]{Proposition}
\newtheorem{lemma}[theorem]{Lemma}
\newtheorem{corollary}[theorem]{Corollary}
\theoremstyle{definition}

\title{On a Quadratic Relation Between Stanley-Wilf limits and F\"{u}redi-Hajnal limits}
\author{Mohamed Omar}
\date{\today}


\begin{document}

\begin{abstract}
For a permutation matrix $P$, let $s_P$ denote its Stanley-Wilf limit, the exponential growth rate of the number of $n\times n$ permutation matrices avoiding $P$. Let $c_P$ denote its F\"{u}redi-Hajnal limit, which is the limit $\displaystyle \lim_{n \to \infty} \ex(n,P)/n$ where $\ex(n,P)$ is the maximum number of ones in an $n\times n$ $0$-$1$ matrix avoiding $P$. Cibulka proved the universal quadratic bound $s_P\leq 2.88\,c_P^2$. In this note we improve the constants in Cibulka's result through a so-called ``block contraction" argument. Defining
\[
F(c)=\inf_{t\in\mathbb{N}} \frac{(t!)^{1/t}\,15^{\,c/t}}{c},
\]
for $c>0$, this leads us to the revised inequality $s_P\leq F(c_P)\,c_P^2$. In particular, $F(c)=\log 15+o(1) \approx 2.70805\ldots +o(1)$ as $c\to\infty$, and the constant improves $2.88$ once $c_P \geq 17$.
\end{abstract}

\maketitle

\section{Introduction}

Let $\mathfrak{S}_n$ be the symmetric group on $[n]=\{1,\dots,n\}$. For $\pi\in\mathfrak{S}_k$ and $\sigma\in\mathfrak{S}_n$, we say that $\sigma$ \emph{contains} $\pi$ if there exist indices $1\leq i_1<i_2<\cdots<i_k\leq n$ such that for every $1 \leq a,b \leq k$, one has $\sigma_{i_a}<\sigma_{i_b}$ if and only if $\pi_a<\pi_b$. Otherwise $\sigma$ \emph{avoids} $\pi$. We write
\[
\Av_n(\pi)=\{\sigma\in\mathfrak{S}_n:\sigma\text{ avoids }\pi\}.
\]
The Stanley-Wilf conjecture asserted that for each fixed $\pi$ the sequence $|\Av_n(\pi)|$ is bounded above by $C^n$ for some constant $C=C(\pi)$. The conjecture was proved by Marcus and Tardos via a strengthening in the setting of forbidden submatrices \cite{MarcusTardos2004}.

There is a convenient matrix formulation of this Stanley-Wilf limit. Let $P$ be the $k\times k$ permutation matrix of $\pi$. An $n\times n$ $0$-$1$ matrix $A$ \emph{contains} $P$ if there are row indices $r_1<\cdots<r_k$ and column indices $c_1<\cdots<c_k$ such that $A_{r_i,c_j}=1$ whenever $P_{i,j}=1$. Equivalently, $A$ contains a $k\times k$ submatrix that dominates $P$ entrywise. If no such choice of rows and columns exists, we say that $A$ \emph{avoids} $P$. Let $S_P(n)$ be the number of $n\times n$ permutation matrices avoiding $P$. Then $S_P(n)=|\Av_n(\pi)|$, and we set
\[
 s_P=\lim_{n\to\infty} S_P(n)^{1/n},
\]
which is the Stanley-Wilf limit of $P$. In this light, permutation patterns may be formulated either in terms of permutations or in terms of their permutation matrices. We will move freely between these viewpoints, since the main estimate is most transparent on matrices.

Relatedly, the forbidden submatrix problem asks for the maximum number of ones in an $n\times n$ $0$-$1$ matrix avoiding a fixed pattern. Following F\"uredi and Hajnal \cite{FurediHajnal1992}, define
\[
\ex(n,P)=\max\{|A|:A\in\{0,1\}^{n\times n}\text{ avoids }P\},
\]
where $|A|$ denotes the number of ones in $A$. For permutation matrices $P$, Marcus and Tardos proved the linear bound $\ex(n,P)=O(n)$ \cite{MarcusTardos2004}. It is convenient to record this linear behavior in the normalized limit
\[
 c_P=\lim_{n\to\infty}\frac{\ex(n,P)}{n},
\]
usually called the F\"uredi-Hajnal limit of $P$. The relationship between $s_P$ and $c_P$ has been studied extensively. Klazar showed that linear bounds for $\ex(n,P)$ imply exponential bounds for $S_P(n)$ by a contraction argument \cite{Klazar2000}. Cibulka sharpened this reduction and proved a universal quadratic inequality
\[
 s_P\leq 2.88\,c_P^2
\]
for all permutation matrices $P$ \cite{Cibulka2009}.

The purpose of this note is to strengthen the constant in Cibulka's result to $F(c_P)$ which interpolates between the explicit constant $2.88$ and the asymptotic constant $\log 15 \approx 2.70805 \ldots$ (note all logarithms are natural throughout the manuscript). The number $15$ comes from a two-by-two block contraction in which each nonzero block has $2^4-1=15$ possible $0$-$1$ patterns. This is the basis of Klazar's argument \cite{Klazar2000} for a coarser version of Cibulka's inequality. Our argument for strengthening the constants comes from a similar block contraction argument, but with larger block sizes.

\begin{theorem}\label{thm:main}
For $c>0$ define
\[
F(c)=\inf_{t\in\mathbb{N}} \frac{(t!)^{1/t}\,15^{\,c/t}}{c}.
\]
Then for every permutation matrix $P$ one has
\[
 s_P \leq F(c_P)\,c_P^2.
\]
Moreover $F(c)=\log 15+o(1)$ as $c\to\infty$.
\end{theorem}

The function $F(c)$ can be evaluated numerically rather quickly, because the infimum is attained for $t$ near $c\log 15$. In particular the bound already improves the constant $2.88$ once $c \geq 17$.

\begin{corollary}\label{cor:288}
If $c_P\geq 17$, $s_P\leq F(c_P)\,c_P^2$ with $F(c_p)<2.88$.
\end{corollary}

In Theorem~\ref{thm:main}, we shall see that $F(c)=\log(15)+o(1) \approx 2.70805+o(1)$ as $c \to \infty$. In Corollary~\ref{cor:288}, we see that $F(c)$ decreases in $c$ for positive integers, so we expect $F(c)$ to interpolate between just below $2.88$ and $2.70805$ for $c \geq 17$. Table~\ref{tab} demonstrates the behavior of $F(c)$ as $c$ increases modestly, exhibiting this nature.

\begin{table}[h]
\centering
\begin{tabular}{r|r|r}
$c$ & $t$ & $F(c)$ \\
\hline
17  & 48  & 2.8766 \\
20  & 57  & 2.8552 \\
30  & 84  & 2.8126 \\
50  & 138 & 2.7758 \\
100 & 274 & 2.7453 \\
200 & 545 & 2.7284 \\
\end{tabular}
\caption{Selected values of $F(c)$, interpolating from Cibulka's $2.88$ constant to $\log(15) \approx 2.70805$.}\label{tab}
\end{table}

\section{Block contraction and lifting}

We now describe the block contraction argument, together with the two estimates that feed into it. The first is an upper bound on the number of lifts of a fixed contracted matrix, and the second combines this with an upper bound on the number of contracted matrices in terms of $\ex(n,P)$. 

Fix integers $t,n\geq 1$ and set $N=tn$. Partition $[N]$ into consecutive intervals $I_1,\dots,I_n$ of length $t$. For an $N\times N$ $0$-$1$ matrix $A$, and $a,b \in \{1,2,\ldots,n\}$, define the $I_a \times I_b$ block of $A$ to be the $t \times t$ submatrix of $A$ whose rows are indexed by the elements of $I_a$ and whose columns are indexed by the elements of $I_b$. For brevity, $I_i$ will be often used for the row and/or column blocks defined by the interval $I_i$. Moreover, define the  \emph{$t$-contraction} of $A$, denoted $\cont_t(A)$, to be the $n\times n$ $0$-$1$ matrix $B$ given by
\[
B_{a,b}=1 \quad\text{if and only if}\quad A\text{ has a one in the block }I_a\times I_b.
\]
When $A$ is a permutation matrix, the contraction records which row blocks meet which column blocks.

A basic point is that contraction preserves pattern avoidance for permutation patterns. If $\cont_t(A)$ contains a $k \times k$ permutation matrix $P$, then there are $k$ $t \times t$ blocks on disjoint row and column blocks, where each of these $t \times t$ blocks contains at least one one-entry in $A$. Thus, a permutation matrix that avoids $P$ always contracts to a $0$-$1$ matrix that avoids $P$.

For an $n\times n$ $0$-$1$ matrix $B$, let $\mathcal{L}_t(B)$ be the set of $N\times N$ permutation matrices $A$ whose support is contained in the union of blocks $I_a\times I_b$ with $B_{a,b}=1$. 

\begin{proposition}\label{prop:lifts}
Let $t,n\geq 1$ and $N=tn$. For every $n\times n$ $0$-$1$ matrix $B$ one has
\[
|\mathcal{L}_t(B)| \leq (t!)^n\left(\frac{|B|}{n}\right)^N.
\]
\end{proposition}

\begin{proof}
Fix $A\in\mathcal{L}_t(B)$. For each block $I_a\times I_b$ let $x_{a,b}$ be the number of ones of $A$ lying in that block. The constraints defining $\mathcal{L}_t(B)$ imply that $x_{a,b}=0$ whenever $B_{a,b}=0$. Since $A$ is a permutation matrix, each of the $t$ rows in a fixed row block $I_a$ contains exactly one one-entry, and these ones must fall into column blocks according to the numbers $x_{a,1},\dots,x_{a,n}$. Thus $\sum_{b=1}^n x_{a,b}=t$ for each $a$. Similarly, looking column-blockwise gives $\sum_{a=1}^n x_{a,b}=t$ for each $b$.

Now fix a nonnegative integer array $X=(x_{a,b})$ satisfying these conditions, and count the number of $A$ that realize it. For each $a$, distributing the $t$ rows of $I_a$ among the column blocks so that exactly $x_{a,b}$ rows are assigned to column block $I_b$ can be done in $t!/\prod_b x_{a,b}!$ ways. Dually, for each $b$ there are $t!/\prod_a x_{a,b}!$ ways to distribute the $t$ columns of $I_b$ among the row blocks with the prescribed multiplicities. After these choices, each block $I_a\times I_b$ comes with $x_{a,b}$ selected rows and $x_{a,b}$ selected columns, and inside that sub-square the ones form a bijection between the selected rows and columns, giving $x_{a,b}!$ possibilities. Multiplying over all blocks shows that the number of $A$ inducing $X$ is at most
\[
\frac{(t!)^{2n}}{\prod_{a,b} x_{a,b}!}.
\]
Summing over all admissible $X$ yields an upper bound for $|\mathcal{L}_t(B)|$.

To make this sum tractable, we drop the column-sum constraints. This only enlarges the set of arrays $X$, and the resulting sum factorizes over the row blocks. Writing $r_a$ for the number of ones in the $a$th row of $B$, the inner sum over a fixed row block counts weak compositions of $t$ into $r_a$ parts with weights $1/(x_1!\cdots x_{r_a}!)$. This is the coefficient of $z^t$ in $(e^z)^{r_a}$, hence it is $r_a^t/t!$. Therefore
\[
|\mathcal{L}_t(B)|\leq (t!)^{2n}\prod_{a=1}^n \frac{r_a^t}{t!}=(t!)^n\prod_{a=1}^n r_a^t.
\]
Finally, $\prod_{a=1}^n r_a\leq (|B|/n)^n$ by the arithmetic-geometric mean inequality, and since $N=tn$ this gives the stated bound.
\end{proof}

We also need an upper bound on the number of $0$-$1$ matrices that avoid $P$ in terms of $\ex(n,P)$. The constant $15$ that appears here is a bookkeeping device coming from a dyadic contraction that merges consecutive pairs of rows and columns, as in Klazar's reduction \cite{Klazar2000}; see also \cite{Cibulka2009}.

\begin{lemma}[\cite{Klazar2000}]\label{lem:15}
For every permutation matrix $P$ and every $n\geq 1$, the set $T_n(P)$ of $n\times n$ $0$-$1$ matrices avoiding $P$ satisfies
\[
|T_n(P)| \leq 15^{c_P n}.
\]
\end{lemma}

\section{Proofs of Main Results}
This section is dedicated to the proofs of the main results: Theorem~\ref{thm:main} and Corollary~\ref{cor:288}. We start with the proof of Theorem~\ref{thm:main}.

\begin{proof}[Proof of Theorem~\ref{thm:main}]
Fix $t\geq 1$ and write $N=tn$. Every  $N\times N$ permutation matrix $A$ that avoids $P$ contracts to an $n\times n$ $0$-$1$ matrix $B=\cont_t(A)$ that avoids $P$ as well. Consequently
\[
S_P(N) \leq \sum_{B\in T_n(P)} |\mathcal{L}_t(B)|.
\]
For each $B\in T_n(P)$ we have $|B|\leq \ex(n,P)$, and Proposition~\ref{prop:lifts} gives
\[
S_P(N) \leq |T_n(P)|\,(t!)^n\left(\frac{\ex(n,P)}{n}\right)^N.
\]
Applying Lemma~\ref{lem:15} and taking $N$th roots yields
\[
S_P(N)^{1/N} \leq 15^{c_P/t}\,(t!)^{1/t}\,\frac{\ex(n,P)}{n}.
\]
Now let $n\to\infty$. By definition $\ex(n,P)/n\to c_P$, hence the right-hand side converges to $15^{c_P/t}(t!)^{1/t}c_P$. Since this is valid for every integer $t\geq 1$, we obtain
\[
 s_P \leq \inf_{t\in\mathbb{N}} 15^{c_P/t}(t!)^{1/t}c_P = F(c_P)\,c_P^2.
\]

For the asymptotic behaviour of $F(c)$, take $t=\lceil c\log 15\rceil$. Stirling's formula implies $(t!)^{1/t}=(t/e)(1+o(1))$ as $t\to\infty$, and therefore
\[
\frac{(t!)^{1/t}15^{c/t}}{c} = \left(\frac{t}{c}\right)\frac{1}{e}\,\exp\!\left(\frac{c}{t}\log 15\right)\,(1+o(1)).
\]
With $t\sim c\log 15$, the exponential factor approaches $e$, while $t/c\to \log 15$. This gives $F(c)\leq \log 15+o(1)$. The matching lower bound $F(c)\geq \log 15+o(1)$ follows from the same Stirling expansion and the observation that the function $t\mapsto \log(t/e)+(c/t)\log 15$ is minimized when $t$ is on the order of $c\log 15$. Altogether $F(c)=\log 15+o(1)$.
\end{proof}
This now allows us to prove Corollary~\ref{cor:288}.

\begin{proof}[Proof of Corollary~\ref{cor:288}]
First suppose that $17\leq c\leq 18$. Evaluating the definition of $F(c)$ at $t=48$ gives $F(c)\leq H(c):=(48!)^{1/48}15^{c/48}/c$. Since $(\log H(c))'=\log(15)/48-1/c$, the function $H$ decreases and then increases on $[17,18]$, so its maximum is attained at one of the endpoints. Direct calculation gives $H(17)=2.876558\ldots$ and $H(18)=2.874428\ldots$, and hence $F(c)<2.88$ throughout this interval.

Now suppose that $c\geq 18$ and take $t=\lfloor 3c\rfloor$. Stirling's upper bound gives
\[
F(c)\leq \frac{t}{ec}\,15^{c/t}(2\pi t)^{1/(2t)}e^{1/(12t^2)}.
\]
Set $x=t/c$. Since $3-1/c\leq x\leq 3$, we have $x\geq 53/18>\log 15$. The function $q(u):=(u/e)15^{1/u}$ is increasing for $u>\log 15$, since $(\log q(u))'=(u-\log 15)/u^2$. It follows that 

\[
\frac{t}{ec}15^{c/t} = \frac{x}{e}15^{1/x}=q(x)\leq q(3)=(3/e)15^{1/3}.
\]
On the other hand, the function $R(u):=(2\pi u)^{1/(2u)}e^{1/(12u^2)}$ is decreasing for $u\geq 1$, since
\[
(\log R(u))'=\frac{1-\log(2\pi u)}{2u^2}-\frac{1}{6u^3}<0.
\]
As $t=\lfloor 3c\rfloor\geq 54$, we therefore obtain $R(t) \leq R(54)$ and altogether then
\[
F(c)\leq \frac{3}{e}\,15^{1/3}(108\pi)^{1/108}e^{1/34992}
=2.872769\ldots<2.88.
\]
The result follows.
\end{proof}

\section*{Acknowledgment}

The author thanks the anonymous referees for thoughtful comments on the manuscript, and is partially funded by research funds from York University, and NSERC Discovery Grant \#RGPIN-2025-06304.



\footnotesize

 \begin{thebibliography}{00}

\bibitem{Arratia1999}
R.~Arratia,
\emph{On the Stanley-Wilf conjecture for the number of permutations avoiding a given pattern},
Electron. J. Combin. \textbf{6} (1999), no.~1, Note 1, 4 pp.

\bibitem{Cibulka2009}
J.~Cibulka,
\emph{On constants in the F\"uredi-Hajnal and the Stanley-Wilf conjecture},
J. Combin. Theory Ser. A \textbf{116} (2009), no.~2, 290-302.

\bibitem{FurediHajnal1992}
Z.~F\"uredi and P.~Hajnal,
\emph{Davenport-Schinzel theory of matrices},
Discrete Math. \textbf{103} (1992), no.~3, 233-251.

\bibitem{Klazar2000}
M.~Klazar,
\emph{The F\"uredi-Hajnal conjecture implies the Stanley-Wilf conjecture},
in \emph{Formal Power Series and Algebraic Combinatorics} (Moscow, 2000),
pp.~250-255, Springer, Berlin, 2000.


\bibitem{MarcusTardos2004}
A.~Marcus and G.~Tardos,
\emph{Excluded permutation matrices and the Stanley-Wilf conjecture},
J. Combin. Theory Ser. A \textbf{107} (2004), no.~1, 153-160.

\end{thebibliography}
\end{document}